\documentclass{amsart}
\usepackage[all,ps,cmtip]{xy}
\usepackage{amsmath, amssymb}
\usepackage{amscd}

\newtheorem{theorem}{Theorem}[section]
\newtheorem{lemma}[theorem]{Lemma}

\newtheorem{corollary}[theorem]{Corollary}

\theoremstyle{definition}
\newtheorem{definition}[theorem]{Definition}

\newtheorem{question}[theorem]{Question}

\theoremstyle{remark}
\newtheorem{remark}[theorem]{Remark}

\numberwithin{equation}{section}

\DeclareMathOperator{\rank}{rk}

\DeclareMathOperator{\id}{id}

\DeclareMathOperator{\td}{td}

\DeclareMathOperator{\codim}{codim}

\DeclareMathOperator{\Hom}{Hom}

\DeclareMathOperator{\spec}{Spec}

\DeclareMathOperator{\Coh}{Coh}

\DeclareMathOperator{\D}{D}

\DeclareMathOperator{\K}{K}

\DeclareMathOperator{\Spec}{Spec}

\DeclareMathOperator{\ch}{ch}

\DeclareMathOperator{\chr}{char}

\begin{document}
\title{Bogomolov's inequality for product type varieties in positive characteristic}

\author{Hao Max Sun}

\address{Department of Mathematics, Shanghai Normal University, Shanghai 200234, People's Republic of China}

\email{hsun@shnu.edu.cn, hsunmath@gmail.com}



\subjclass[2000]{Primary 14J60: Secondary 14G17, 14F05}

\date{May 16, 2019}

\keywords{Bogomolov's inequality, semistable sheaf, Hilbert
stability, Bridgeland stability condition, positive characteristic}

\begin{abstract}
We prove Bogomolov's inequality for semistable sheaves on product
type varieties in arbitrary characteristic. This gives the first
examples of varieties of general type in positive characteristic on
which Bogomolov's inequality holds for semistable sheaves of any
rank. The key ingredient in the proof is a high rank generalization
of the slope inequality established by Xiao and Cornalba-Harris.
This Bogomolov's inequality is applied to study the positivity of
linear systems and semistable sheaves and construct Bridgeland
stability conditions on product type surfaces in positive
characteristic. We also give some new counterexamples to Bogomolov's
inequality and pose some open questions.
\end{abstract}

\maketitle

\setcounter{tocdepth}{1}
\tableofcontents

\section{Introduction}
Throughout this paper, we fix an algebraically closed field $k$ of
arbitrary characteristic. Let $X$ be a smooth projective variety
defined over $k$ with $\dim X\geq2$, and let $H$ be an ample divisor
on $X$. The famous Bogomolov's inequality says that if $\chr(k)=0$,
then
$$\Delta(E)H^{\dim X-2}=(\ch^2_1(E)-2\ch_0(E)\ch_2(E))H^{\dim X-2}\geq0,$$ for any
$\mu_H$-semistable sheaf $E$ on $X$. It was proved by Bogomolov
\cite{Bog} when $\dim X=2$, and it can be easily generalized to
higher dimensional case by the Mehta-Ramanathan restriction theorem.

In the case of $\chr(k)>0$, Langer \cite{Langer1} proved that the
same inequality holds for strongly $\mu_H$-semistable sheaves. Mehta
and Ramanathan \cite{MR} showed that if $X$ satisfies
$\mu_H^+(\Omega_X^1)\leq0$, then all $\mu_H$-semistable sheaves on
$X$ are strongly $\mu_H$-semistable. Thus Bogomolov's inequality
holds on such an $X$. One notices that the Kodaira dimension of this
$X$ is non-positive.

In general it is well known that Bogomolov's inequality fails for
semistable sheaves in positive characteristic. And it is only known
to be held for some special varieties. For example, Shepherd-Barron
\cite{SB} proved that Bogomolov's inequality holds for rank two
semistable sheaves on surfaces which are neither quasi-elliptic with
$\kappa(X)=1$ nor of general type. This result was generalized by
Langer \cite[Theorem 7.1]{Langer3} to the higher-rank case. And
Langer \cite{Langer2} also showed this inequality holds for any
semistable sheaf $E$ with $\rank E\leq\chr(k)$ on a variety which
can be lifted to the ring of Witt vectors of length 2. In this
paper, we prove that Bogomolov's inequality holds for semistable
sheaves of arbitrary rank on product type varieties in any
characteristic. It gives the first examples of varieties of general
type in positive characteristic on which Bogomolov's inequality
holds for semistable sheaves of any rank.

\begin{definition}
Let $X$ be a smooth projective variety defined over $k$.
\begin{enumerate}
\item We say that $X$ is a product type variety if there exist smooth
projective curves $C_1$,$\cdots$, $C_n$ defined over $k$ and a
finite separable surjective morphism $f:C_1\times\cdots\times
C_n\rightarrow X$.

\item A divisor $H$ on the product type variety $X$ is called a product type ample divisor if
$f^*H$ can be written as $f^*H=p_1^*A_1+\cdots+p_n^*A_n$ for some
ample divisors $A_i$ on $C_i$, where $p_i:C_1\times\cdots\times
C_n\rightarrow C_{i}$ is the projection for $i=1,\cdots,n$.
\end{enumerate}
\end{definition}
A simple example of product type varieties is the symmetric product
of a curve. The varieties isogenous to a product of curves
introduced by Catanese \cite{Cat} are other important examples. See
also \cite{BP, FGP} for a huge number of interesting examples called
product-quotient varieties. The product-quotient varieties are our
product type varieties if they are smooth. Our main results are the
following two theorems.
\begin{theorem}\label{main}
Let $X$ be a product type variety of dimension $n$ and $H$ a product
type ample divisor on $X$. Then for any $\mu_{H}$-semistable sheaf
$E$ on $X$, we have
$$H^{n-2}\Delta(E)\geq0.$$
\end{theorem}
The product type assumption on $H$ can be dropped when $n=2$:

\begin{theorem}\label{Surface}
Let $S$ be a smooth projective surface birational to a product type
surface and $H$ a numerically nontrivial nef divisor on $S$. Then
for any $\mu_{H}$-semistable sheaf $E$ on $S$, we have
$\Delta(E)\geq0$.
\end{theorem}
In characteristic zero there are several proofs of Bogomolov's
inequality. The first proof is due to Bogomolov \cite{Bog}. The key
ingredient in his proof is that the tensor power of a semistable
vector bundle is still semistable. The second one is given by
Gieseker \cite{Gieseker} using reduction mod $p$ and estimating the
dimension of the space of sections of the Frobenius pull back of a
semistable sheaf. The third proof is transcendental, using the
Kobayashi-Hitchin correspondence between the polystability and the
existence of Hermite-Einstein metric (see \cite{Kob}).

Unfortunately, all these proofs do not work in positive
characteristic. And our proof of the above theorems is totally
different from theirs. The strategy of the proof is the following.
Firstly, we use a result of \cite{Sch} to show the equivalence
between semistability and Hilbert stability for locally free sheaves
on curves. Then we generalize the method of Cornalba-Harris
\cite{CH} to prove a high rank slope inequality for relative
semistable sheaves on a fibration (Theorem \ref{Slope}). This
inequality implies that Moriwaki's relative Bogomolov's inequality
\cite{Mo1} holds for trivial fibrations in any characteristic
(Corollary \ref{relative}). By the techniques of the changes of
polarizations in \cite[Appendix 4.C]{HL}, one obtains our main
theorems.

We exhibit the strategy of our proof in the following chain of
implications.
\begin{eqnarray*}
&&\boxed{\mbox{Semistability}}\Rightarrow\boxed{\mbox{Hilbert
stability}}\Rightarrow \boxed{\mbox{High rank slope inequality}}\\
&&\Rightarrow \boxed{\mbox{Relative Bogomolov's
inequality}}\Rightarrow \boxed{\mbox{Bogomolov's inequality}}
\end{eqnarray*}

\subsection*{Applications and open questions}
Bogomolov's inequality has many interesting applications, such as
the the positivity of adjoint linear systems (see \cite{Reider,
BS}), and vanishing theorems for semistable sheaves (see
\cite{Sun}). By Theorem \ref{Surface}, all these related results
automatically hold for product type surfaces in positive
characteristic (see Theorem \ref{Reider} and \ref{Sun}).

The authors in \cite{Bri2,AB} showed that Bogomolov's inequality for
semistable sheaves of any rank can be used to construct Bridgeland
stability conditions on surfaces. Hence Bridgeland stability
condition always exists on surfaces in characteristic zero. It is
natural to ask:
\begin{question}\label{Stab}
For a smooth projective surface in positive characteristic, is there
any Bridgeland stability condition on it?
\end{question}

Theorem \ref{Surface} can give an affirmative answer of this
question for product type surfaces (see Theorem \ref{Bri}). They are
the first examples of Bridgeland stability conditions on some
surfaces of general type in positive characteristic. Because of the
existence of counterexamples to Bogomolov's inequality, it seems
that a positive answer of Question \ref{Stab} needs a completely new
construction.

Theorem \ref{main} and \ref{Surface} inspire us to construct some
new counterexamples to Bogomolov's inequality (Theorem \ref{Count}).
Langer's \cite[Theorem 1]{Langer2} and our counterexamples lead us
to pose the below question:
\begin{question}\label{Conj}
Let $S$ be a smooth projective minimal surface defined over $k$ with
$c_1(S)^2\leq2c_2(S)$. Let $H$ be an ample divisor on $S$. Assume
that $S$ can be lifted to the ring $W_2(k)$ of Witt vectors of
length $2$. Does $\Delta(E)\geq0$ hold for any $\mu_H$-semistable
torsion free sheaf $E$?
\end{question}
We notice the condition $c_1^2\leq2c_2$ in Question \ref{Conj} is
satisfied for surfaces isogenous to a product of curves.

\subsection*{Organization of the paper}
Our paper is organized as follows. In Section \ref{S2}, we review
basic notions and properties of the classical stability for coherent
sheaves. Some results of Butler \cite{But} and M. Teixidor
\cite{Tei1} have been generalized to any characteristic which may be
of interest in other contexts. Then in Section \ref{S3} we recall
the definition and properties of Hilbert stability for locally free
sheaves on curves. In Section \ref{S4} we show the high rank slope
inequality for relative semistable sheaves (Theorem \ref{Slope}) and
the relative Bogomolov's inequality for trivial fibrations
(Corollary \ref{relative}). We prove Theorem \ref{main} and
\ref{Surface} in Section \ref{S5}. The applications of our main
theorems will be discussed in Section \ref{S6} (Theorem
\ref{Reider}, \ref{Sun} and \ref{Bri}). In Section \ref{S7}, we give
new counterexamples to Bogomolov's inequality in positive
characteristic (Theorem \ref{Count}).

\subsection*{Notation}
We work over an algebraically closed field $k$ of arbitrary
characteristic in this paper, and write $\chr(k)$ for its
characteristic. Let $X$ be a smooth projective variety defined over
$k$. We denote by $\D^b(X)$ its bounded derived category of coherent
sheaves and by $\kappa(X)$ its Kodaira dimension. $K_X$ and
$\omega_X$ denote the canonical divisor and canonical sheaf of $X$,
respectively. When $\dim X=1$, we write $g(X)$ for the genus of the
curve $X$. For a morphism $f:X\rightarrow Y$ of smooth varieties, we
denote by $K_{X/Y}$ the relative canonical divisor $K_X-f^*K_Y$ of
$f$. For a triangulated category $\mathcal{D}$, we write
$\K(\mathcal{D})$ for the Grothendieck group of $\mathcal{D}$.

We write $\ch(E)$ and $c(E)$ for the Chern character and Chern class
of a complex $E\in \D^b(X)$, respectively. We also write $H^j(F)$
($j\in \mathbb{Z}_{\geq0}$) for the cohomology groups of a sheaf
$F\in\Coh(X)$ and $h^j(F)$ for the dimension of $H^j(F)$. For a
sheaf $G\in\Coh(X)$, we denote by $G^*:=\Hom(G, \mathcal{O}_X)$ the
dual sheaf of $G$ and by $\Delta(G):=\ch^2_1(G)-2\ch_0(G)\ch_2(G)$
the discriminant of $G$. Given a complex number $z\in\mathbb{C}$, we
denote its real and imaginary part by $\Re z$ and $\Im z$,
respectively.

\subsection*{Acknowledgments}
The author is grateful to the referee for his or her valuable
comments and pointing out some errors in the earlier version of this
paper. The author would also like to thank Rong Du, Lingguang Li,
Xin L\"u, Xiaotao Sun, Wanyuan Xu, Fei Yu and Lei Zhang for their
interest and discussions. The author was supported by National
Natural Science Foundation of China (Grant No. 11771294, 11301201).

\section{Slope stability for sheaves}\label{S2}
In this section, we will review some basic properties of slope
stability for coherent sheaves, and generalize the tensor product
theorem and Butler's theorem to arbitrary characteristic.

Let $X$ be a smooth projective variety with $\dim X=n$ defined over
$k$. Let us fix a collection of nef divisors $H_1,\cdots,H_{n-1}$ on
$X$. We define the slope $\mu_{H_1,\cdots,H_{n-1}}$ of a coherent
sheaf $E\in \Coh(X)$ by
\begin{eqnarray*}
\mu_{H_1,\cdots,H_{n-1}}(E)= \left\{
\begin{array}{lcl}
+\infty,  & &\mbox{if}~\rank(E)=0,\\
&&\\
\frac{H_1\cdots H_{n-1}\ch_1(E)}{\rank(E)}, & &\mbox{otherwise}.
\end{array}\right.
\end{eqnarray*}
We write $\mu$ for $\mu_{H_1,\cdots,H_{n-1}}$ if there is no
confusion. When $H_1=\cdots=H_{n-1}=H$, we also write $\mu_H$ for
$\mu_{H_1,\cdots,H_{n-1}}$.

\begin{definition}\label{def2.1}
A coherent sheaf $E$ on $X$ is $\mu$-semistable (or semistable) if,
for all non-zero subsheaves $F\hookrightarrow E$, we have
$$\mu(F)\leq\mu(E).$$
\end{definition}

Notice that a torsion sheaf is always $\mu$-semistable and a
$\mu$-semistable sheaf is either a torsion free sheaf or a torsion
sheaf. Harder-Narasimhan filtrations (HN-filtrations, for short)
with respect to $\mu$-stability exist in $\Coh(X)$: given a non-zero
sheaf $E\in\Coh(X)$, there is a filtration
$$0=E_0\subset E_1\subset\cdots\subset E_m=E$$
such that: $G_i:=E_i/E_{i-1}$ is $\mu$-semistable, and
$\mu(G_1)>\cdots>\mu(G_m)$. We set $\mu^+(E):=\mu(G_1)$ and
$\mu^-(E):=\mu(G_m)$.

\begin{lemma}\label{lemma2.2}
If $E$ and $F$ are torsion free sheaves on $X$ with
$\mu^-(E)>\mu^+(F)$, then $\Hom(E,F)=0$.
\end{lemma}
\begin{proof}
See \cite[Lemma 1.3.3]{HL}.
\end{proof}

\subsection{Tensor product theorem}\label{S2.1}
Let $C$ be a smooth projective curve defined over $k$. It is well
known that if $\chr(k)=0$, then $$\mu^{\pm}(E\otimes
F)=\mu^{\pm}(E)+\mu^{\pm}(F),$$ for any locally free sheaves $E$ and
$F$ on $C$ (\cite[Lemma 2.5]{But}). This result fails when
$\chr(k)>0$ (see \cite{Gie} for counterexamples). However, we have
the following theorem:
\begin{theorem}\label{tensor}
Let $E$ and $F$ be locally free sheaves on $C$. Then
\begin{enumerate}
\item $0\leq\mu^+(E\otimes
F)-\mu^+(E)-\mu^+(F)\leq g(C)$;

\item $-g(C)\leq\mu^-(E\otimes
F)-\mu^-(E)-\mu^-(F)\leq0$;

\item $\mu^-(\otimes^mE)\geq
m\mu^-(E)-(m-1)g(C)$, for any positive integer $m$.
\end{enumerate}
\end{theorem}
\begin{proof}
It is clear that $\mu^+(E\otimes F)\geq \mu^+(E)+\mu^+(F)$. One
needs to show $$\mu^+(E\otimes F)\leq \mu^+(E)+\mu^+(F)+g(C).$$ Set
$l$ be the smallest integer that is greater than or equal to
$-\mu^+(E)-\mu^+(F)-1$. Take a line bundle $L$ on $C$ with $\deg
L=l$. One sees that
\begin{equation}\label{2.1}
-1\leq\mu^+(E)+\mu^+(F)+l=\mu^+(E)+\mu^+(F\otimes L)<0.
\end{equation}

We will prove that $\mu^+(E\otimes F\otimes L)\leq g(C)-1$. Let $G$
be a subsheaf of $E\otimes F\otimes L$ such that
$\mu(G)=\mu^+(E\otimes F\otimes L)$. By the Riemann-Roch theorem, it
follows that $h^0(G)\geq \deg G-\rank G(g(C)-1)$. Hence, if
$\mu^+(E\otimes F\otimes L)> g(C)-1$, we have $$\frac{h^0(G)}{\rank
G}\geq \mu(G)-g(C)+1>0.$$ This implies $$\hom(E^*, F\otimes L)=
h^0(E\otimes F\otimes L)\geq h^0(G)>0.$$ By Lemma \ref{lemma2.2},
one obtains $\mu^-(E^*)\leq\mu^+(F\otimes L)$. Then
$\mu^+(E)+\mu^+(F\otimes L)\geq0$, since $\mu^+(E)=-\mu^-(E^*)$. It
contradicts (\ref{2.1}). Thus $$\mu^+(E\otimes F\otimes L)\leq
g(C)-1\leq g(C)+\mu^+(E)+\mu^+(F\otimes L).$$ This implies the first
conclusion.

Since $\mu^-(E)=-\mu^+(E^*)$, $\mu^-(F)=-\mu^+(F^*)$ and
$\mu^-(E\otimes F)=-\mu^+(E^*\otimes F^*)$, one can immediately get
(2) from (1). The assertion (3) follows from (2) by induction on
$m$.
\end{proof}

\subsection{Butler's theorem}\label{S2.2}
The following is a generalization of \cite[Theorem 2.1]{But} to
arbitrary characteristic.

\begin{theorem}\label{Butler}
Let $E$ and $F$ be locally free sheaves on a smooth projective curve
$C$ defined over $k$. Assume $\mu^-(E)\geq3g(C)$ and
$\mu^-(F)\geq3g(C)$. Then the multiplication map $$H^0(E)\otimes
H^0(F)\rightarrow H^0(E\otimes F).$$ is surjective.
\end{theorem}
\begin{proof}
Since $\mu^-(E)\geq3g(C)$, by \cite[Lemma 1.12]{But}, one sees that
$E$ is generated by global sections. Hence the evaluation map of $E$
determines an exact sequence:
\begin{equation}\label{2.2}
0\rightarrow M_E\rightarrow H^0(E)\otimes\mathcal{O}_C\rightarrow
E\rightarrow0.
\end{equation}
From \cite[Corollary 1.3]{But}, it follows that
$$\mu^-(M_E)\geq-\frac{\mu^-(E)}{\mu^-(E)-g(C)}\geq-\frac{3}{2}.$$
By Lemma \ref{tensor}, we have
$$\mu^-(M_E\otimes F)\geq2g(C)-\frac{3}{2}>2g(C)-2.$$
This implies $$h^1(M_E\otimes F)=\hom(M_E\otimes F, \omega_C)=0.$$
Tensoring sequence (\ref{2.2}) by $F$ and taking cohomology proves
the theorem.
\end{proof}

From this, we can deduce the following generalization of a result in
\cite{Tei1}.

\begin{theorem}\label{exterior}
Let $E$ be a semistable locally free sheaf of rank $r$ on a smooth
projective curve $C$ defined over $k$. If $\mu(E)\geq3g(C)$, then
the map $$\wedge^rH^0(E)\rightarrow H^0(\wedge^rE)$$ is surjective.
\end{theorem}
\begin{proof}
From Theorem \ref{tensor} and \ref{Butler}, one infers that the map
$$\otimes^rH^0(E)\rightarrow H^0(\otimes^rE)$$ is surjective. On the
other hand, the canonical map $\tau:\otimes^rE\rightarrow \wedge^rE$
gives rise to an exact Koszul complex: $$\cdots\rightarrow
\wedge^2(\otimes^rE)\otimes\wedge^rE^*\rightarrow
\otimes^rE\stackrel{\tau}{\longrightarrow} \wedge^rE\rightarrow0.$$
Since $\wedge^2(\otimes^rE)\otimes\wedge^rE^*$ is a quotient of
$(\otimes^{2r}E)\otimes\wedge^rE^*$, one sees that
\begin{eqnarray*}
\mu^-\left(\wedge^2(\otimes^rE)\otimes\wedge^rE^*\right)
&\geq&\mu^-\left((\otimes^{2r}E)\otimes\wedge^rE^* \right)\\
&=&\mu^-(\otimes^{2r}E)-\deg E\\
&\geq&2r\mu^-(E)-(2r-1)g(C)-\deg E\\
&=&r\mu(E)-(2r-1)g(C)\\
&\geq&(r+1)g(C)\\
&\geq&2g(C).
\end{eqnarray*}
This infers $H^1(\wedge^2(\otimes^rE)\otimes\wedge^rE^*)=0$. Thus
$H^1(\ker \tau)=0$ and the map $$H^0(\otimes^rE)\rightarrow
H^0(\wedge^rE)$$ is surjective. It follows that the composite map
$$\otimes^rH^0(E)\rightarrow H^0(\wedge^rE)$$ is also surjective. By the universal
property of the exterior algebra, we obtain the desired surjection.
\end{proof}

\begin{remark}
One sees that the bound $3g(C)$ in Theorem \ref{Butler} can be
slightly improved by its proof. But we don't need this.
\end{remark}

\section{Hilbert stability for locally free sheaves}\label{S3}
Throughout this section, we let $E$ be a semistable locally free
sheaf on a smooth projective curve $C$ defined over $k$ with $\deg
E=d$ and $\rank E=r$. We further assume that $\mu(E)\geq 3g(C)$ and
set $V=H^0(E)$. We will recall the definition and some basic
properties of Hilbert stability for such an $E$ on $C$.

By our assumptions, one sees that the evaluation map $V\otimes
\mathcal{O}_C\rightarrow E$ is surjective, and it defines a morphism
$$f:C\rightarrow \mathbb{G}(V, r),$$ here $\mathbb{G}(V, r)$ is the
Grassmannian of $r$ dimensional quotients of $V$. Let $$p:
\mathbb{G}(V, r)\hookrightarrow \mathbb{P}(\wedge^rV)$$ be the
Pl\"ucker embedding, where $\mathbb{P}(\wedge^rV)$ is the projective
space of $1$ dimensional quotients of $\wedge^rV$.

\begin{lemma}
The morphisms $f$ and $p\circ f$ are embeddings.
\end{lemma}
\begin{proof}
Note that the morphism $p\circ f: C\rightarrow\mathbb{P}(\wedge^rV)$
is defined by the surjection
$$\wedge^rV\otimes \mathcal{O}_C\rightarrow
\wedge^rE.$$ Since $\mu(E)\geq 3g(C)$, one sees $\wedge^r E$ is a
very ample line bundle, and the map
$$\wedge^rV\rightarrow H^0(\wedge^rE)$$ is
surjective by Theorem \ref{exterior}. We deduce that $p\circ f$ is
an embedding. Thus $f$ is an embedding.
\end{proof}

We let $\mathcal{I}_C$ be the ideal sheaf of $C$ in
$\mathbb{P}(\wedge^rV)$. Since $H^1(\mathcal{I}_C(m))=0$ for $m$
large enough, from the short exact sequence
$$0\rightarrow\mathcal{I}_C(m)\rightarrow\mathcal{O}_{\mathbb{P}(\wedge^rV)}(m)\rightarrow\mathcal{O}_C(m)\rightarrow0,$$
one sees that the map
$H^0(\mathcal{O}_{\mathbb{P}(\wedge^rV)}(m))\rightarrow
H^0(\mathcal{O}_C(m))$ is surjective. Hence we obtain a surjection
$$\psi_m:S^m(\wedge^rV)\rightarrow H^0((\det E)^{\otimes m}).$$
Let $P(m)=h^0((\det E)^{\otimes m})=dm-g(C)+1$. One finally obtains
a map
$$\varphi_m:\wedge^{P(m)}S^m(\wedge^rV)\rightarrow
\wedge^{P(m)}H^0((\det E)^{\otimes m})\cong k.$$ It gives a point
$$[\varphi_m]\in \mathbb{P}\left(\wedge^{P(m)}
S^m(\wedge^rV)\right)$$

\begin{definition}
We say $(C,E)$ is $m$-Hilbert stable (resp., semistable) if the
point $[\varphi_m]$ is stable (resp., semistable) under the induced
action of $SL(V)$, i.e., $[\varphi_m]$ has closed orbit and finite
stabilizer (resp., $0$ is not in the closure of the orbit of
$[\varphi_m]$). We say $(C,E)$ is Hilbert stable (resp., semistable)
if it is $m$-Hilbert stable (resp., semistable) for all $m$
sufficiently large.
\end{definition}

Recall that a necessary and sufficient condition for the
semistability of $[\varphi_m]$ is the existence of a
$SL(V)$-invariant non-constant homogeneous polynomial $$h\in
S^N\left(\wedge^{P(m)}S^m(\wedge^rV)\right)$$ such that
$(S^N\varphi_m)(h)\neq0$.

\begin{theorem}\label{Schmitt}
There is a constant $d_0=d_0(r, g(C))$ so that for each $d\geq d_0$,
there exists a constant $m_0=m_0(d,r, g(C))$ such that if $m\geq
m_0$, then $(C,E)$ is $m$-Hilbert stable (resp., semistable) if and
only if $E$ is stable (resp., semistable).
\end{theorem}
\begin{proof}
The required conclusion was first proved for rank 2 by Gieseker and
Morrison in \cite[Theorem 1.1]{GM}(see also \cite{Tei2} for a
different proof). For general rank, one of the implications was
proved in \cite[Proposition 2.2]{Tei3}, and the equivalence was
given by Schmitt \cite{Sch} in characteristic zero. The
characteristic 0 assumption in Schmitt's proof is only used in
\cite[Corollary 1.1.2]{Sch} which has been generalized to arbitrary
characteristic case in Theorem \ref{exterior}. Hence Schmitt's proof
works in any characteristic.
\end{proof}

\section{High rank slope inequalities}\label{S4}
In this section, we will prove the high rank slope inequality and
relative Bogomolov's inequality for a trivial fibration. Throughout
this section, we let $\pi: X\rightarrow Y$ be a flat, projective,
and surjective morphism of smooth projective varieties over $k$ with
$\dim X=n$ and $\dim Y=n-1$. Let $y$ be a general point of $Y$. We
further assume that the general fiber $X_y:=X\times_Y\Spec(k(y))$ is
a connected smooth curve of genus $g$. For a sheaf $E$ on $X$, we
write $E_y$ for the restriction of $E$ to $X_y$.

The following is a high rank generalization of \cite[Theorem
1.1]{CH} (see also \cite[Theorem 1.5]{Sto}).

\begin{theorem}\label{CH}
Let $E$ be a reflexive sheaf on $X$ such that $E_y$ is semistable
and $\mu(E_y)\geq3g$. Suppose that $(X_y, E_y)$ is $m$-Hilbert
semistable for some positive integer $m$. Set $\det E=L$, $\rank
E=r$ and $P(m)=\rank\pi_*(L^{m})$. Let $D_m(E)$ be the line bundle
$$(\det\pi_*(L^m))^{\rank \pi_*E}\otimes(\det\pi_*E)^{-P(m)mr}.$$ Then there is a positive integer $N$
such that $(D_m(E))^N$ is effective.
\end{theorem}
\begin{proof}
By the reflexive assumption of $E$, one sees that $E$ is locally
free except along a closed subscheme $W$ of $X$ of codimension
$\geq3$. On the other hand, from \cite[Corollary 1.7]{Hart2}, it
follows that $\pi_*E$ and $\pi_*((\det E)^m)$ are also reflexive
sheaves. Thus both $\pi_*E$ and $\pi_*((\det E)^m)$ are locally free
off a closed subscheme of $Y$ of codimension three or greater. Since
$$\dim \pi(W)\leq\dim W\leq n-3=\dim Y-2,$$ we can
take a closed subscheme $W^\prime$ of $Y$ with $\codim
W^\prime\geq2$ such that $\pi_*E$ and $\pi_*((\det E)^m)$ are
locally free on $U:=Y-W^\prime$ and $E$ is locally free on
$\pi^{-1}(U)$.



We consider the natural morphism
$\gamma^m:S^m(\wedge^r(\pi_*E)|_U)\rightarrow(\pi_*L^m)|_U$. By our
assumptions, one sees that the fibre of $\gamma^m$ at $y$,
$$\gamma^m_y:S^m(\wedge^{r}H^0(E_y))\rightarrow H^0(L_y^m),$$ is
surjective. Set $V=H^0(E_y)$. Since $(X_y, E_y)$ is $m$-Hilbert
semistable, there exists a $SL(V)$-invariant non-constant degree
$N_0$ homogeneous polynomial
$$f\in S^{N_0}\left(\bigwedge^{P(m)}S^m(\wedge^rV)\right)$$ such that
\begin{equation}\label{poly}
0\neq\left(S^{N_0}\bigwedge^{P(m)}\gamma^m_y\right)(f)\in\left(\det
H^0(L_y^m)\right)^{N_0}.
\end{equation}
One sees that the one dimensional linear subspace $W$ generated by
$f$ in $S^{N_0}\left(\wedge^{P(m)}S^m(\wedge^rV)\right)$ is
invariant under the action of $GL(V)$.

Let $$\rho: GL(V)\rightarrow
GL\left(\bigwedge^{P(m)}S^m(\wedge^rV)\right)$$ be the standard
representation and $\sigma: GL(V)\rightarrow GL(W)$ the restriction
representation from $S^{N_0}\rho$. Composing the transition
functions of $(\pi_*E)|_U$ with $\rho$ (resp., $\sigma$), one can
construct a new locally free sheaf $((\pi_*E)|_U)_{\rho}$ (resp.,
$((\pi_*E)|_U)_{\sigma}$) and an injective morphism
$$((\pi_*E)|_U)_{\sigma}\hookrightarrow S^{N_0}((\pi_*E)|_U)_{\rho}.$$ Since
$((\pi_*E)|_U)_{\rho}=\wedge^{P(m)}S^m(\wedge^r(\pi_*E)|_U)$,
composing this injection with $S^{N_0}\bigwedge^{P(m)}\gamma^m$, we
gets a morphism
$\widetilde{\gamma}:((\pi_*E)|_U)_{\sigma}\rightarrow
(\det(\pi_*L^m)|_U)^{N_0}$. From property (\ref{poly}) and our
construction, it follows that $\widetilde{\gamma}$ is non-zero. It
remains to compute $((\pi_*E)|_U)_{\sigma}$ explicitly.

Take an element $A\in GL(V)$. We can write $$A=(\det
A)^{\frac{1}{\dim V}}B,$$ where $B\in SL(V)$. The action of $A$ on
$f$ is given by the following:
\begin{eqnarray*}
\sigma(A)f&=&S^{N_0}\rho((\det A)^{\frac{1}{\dim V}}B)f\\
&=& S^{N_0}\rho\left((\det A)^{\frac{1}{\dim
V}}\id_V\right)\left(S^{N_0}\rho(B)f\right)\\
&=&S^{N_0}\rho\left((\det A)^{\frac{1}{\dim V}}\id_V\right)f\\
&=&(\det A)^{\frac{N_0P(m)mr}{\dim V}}f.
\end{eqnarray*}
Since any character of the algebraic group $GL(V)$ is a power of the
determinant, one sees that $N_0P(m)mr$ is a multiple of $\dim V$. It
follows that
$$((\pi_*E)|_U)_{\sigma}=(\det(\pi_*E)|_U)^{\frac{N_0P(m)mr}{\dim V}}.$$
Hence the line bundle
\begin{eqnarray*}
&&\Big((\det(\pi_*L^m)|_U)^{N_0}\otimes(\det(\pi_*E)|_U)^{-\frac{N_0P(m)mr}{\dim
V}}\Big)^{\dim V}\\
&=&\left((\det(\pi_*L^m)|_U)^{\dim
V}\otimes(\det(\pi_*E)|_U)^{-P(m)mr}\right)^{N_0}\\
&=&D_m(E)^{N_0}|_U
\end{eqnarray*}
is effective. So is $D_m(E)^{N_0}$.
\end{proof}

From Theorem \ref{CH}, we can deduce the following slope inequality
for relative semistable sheaves on $X$. The original slope
inequality is proved by Xiao \cite{Xiao} for the relative canonical
sheaf of a surface fibration in characteristic zero and
independently by Cornalba-Harris \cite{CH} for semi-stable
fibrations. Stoppino \cite{Sto} showed that the method of
Cornalba-Harris still works for non-semistable fibrations. See also
\cite{Mo2, SSZ} for the slope inequality in positive characteristic.

\begin{theorem}\label{Slope}
Let $E$ be a reflexive sheaf of rank $r$ on $X$ such that $E_y$ is
semistable. Let $A_1,\cdots, A_{n-2}$ be ample divisors on $Y$. Then
there exists an integer $d_0$ such that if $\deg E_y\geq d_0$, we
have
$$\pi^*(A_1\cdots A_{n-2})c_1^2(E)\geq\frac{2r\deg E_y}{h^0(E_y)}A_1\cdots A_{n-2}c_1(\pi_*E).$$
\end{theorem}
\begin{proof}
By Theorem \ref{Schmitt}, one sees that there is an integer
$d_0\geq3gr$ so that $(X_y, E_y)$ is $m$-Hilbert semistable when
$\deg E_y\geq d_0$ and $m$ large enough. Therefore, from Theorem
\ref{CH}, we obtain an effective line bundle
$$\left((\det\pi_*(L^m))^{\rank
\pi_*E}\otimes(\det\pi_*E)^{-\rank\pi_*(L^m)mr}\right)^{N},$$ here
$N$ is an positive integer and $L=\det E$. This implies
\begin{equation}\label{4.2}
A_1\cdots A_{n-2}\Big(\rank
(\pi_*E)c_1(\pi_*(L^m))-\rank\pi_*(L^m)mrc_1(\pi_*E)\Big)\geq0.
\end{equation}

On the other hand, by the Grothendieck-Riemann-Roch theorem, one has
the following formula (see \cite[Lemma 2.3]{Mo2} for example):
$$c_1(\pi_*(L^m))=\frac{\pi_*(c_1^2(E))}{2}m^2+Z_1m+Z_0,$$ here
$Z_1$ and $Z_0$ are $\mathbb{Q}$-divisors of $Y$. Moreover, some
simple computations show that $$\rank\pi_*(L^m)=h^0(L_y^m)=m\deg
L_y-g+1$$ and $\rank(\pi_*E)=h^0(E_y)$. Substituting these equations
into (\ref{4.2}) and letting $m\rightarrow+\infty$, we obtain the
desired inequality.
\end{proof}

\begin{corollary}\label{Slope2}
Let $H$ be a $\pi$-relatively ample divisor on $X$, $A_1,\cdots,
A_{n-2}$ ample divisors on $Y$ and $E$ a rank $r$ torsion free sheaf
on $X$. Suppose $E_y$ is semistable. Then we have
$$\pi^*(A_1\cdots A_{n-2})\Big((\deg H_y)HK_{X/Y}-(g-1)H^2\Big)\geq0.$$ If the
equality holds, then
\begin{eqnarray}\label{4.3}
\nonumber\pi^*(A_1\cdots A_{n-2})\Delta(E)&\geq&\pi^*(A_1\cdots A_{n-2})\Big(\frac{r^2}{6}(c_1^2(X)+c_2(X))-rc_1(E)K_{X/Y}\Big)\\
&-&\frac{r\pi^*(A_1\cdots A_{n-2})}{\deg H_y}\Big((\deg
E_y)HK_{X/Y}-(2g-2)Hc_1(E)\Big)\\
\nonumber &-&r^2(g-1)A_1\cdots A_{n-2}K_Y.
\end{eqnarray}
\end{corollary}
\begin{proof}
Since $E$ is torsion free, one sees that the natural map
$E\rightarrow E^{**}$ is injective and the quotient $E^{**}/E$ has
support of codimension $\geq2$. It follows that $$\pi^*(A_1\cdots
A_{n-2})\Delta(E)\geq\pi^*(A_1\cdots A_{n-2})\Delta(E^{**}).$$
Therefore, to prove the corollary, it suffices to deal with the case
that $E$ is reflexive.

Let $S_m$ be the divisor class
$$h^0(E_y(mH))\pi_*c_1^2(E(mH))-2r\deg (E_y(mH))c_1(\pi_*E(mH)).$$
Applying Theorem \ref{Slope} for $E(mH)$, we have $A_1\cdots
A_{n-2}S_m\geq0$ for $m\gg0$. It remains to understand the terms in
$S_m$ explicitly.

By the reflexiveness of $E$, one sees that $E$ is locally free
except along a closed subscheme $W$ of $X$ of codimension $\geq3$.
Let $U=Y-\pi(W)$. Then $E|_{\pi^{-1}(U)}$ is locally free, and thus
it is flat over $U$. This implies $(R^i\pi_*E(mH))|_{U}=0$ for
$i\geq1$ and $m\gg0$, by Grauert's theorem (see \cite[Corollary
12.9]{Hart}). Since $\codim \pi(W)\geq2$, one obtains
$$\ch_j(\pi_*E(mH))=\ch_j(\pi_!E(mH)),$$ for $j=0$ or $1$.
From the Grothendieck-Riemann-Roch theorem, it follows that
$$\ch(\pi_!E(mH))\td(Y)=\pi_*\Big(\ch(E(mH))\td(X)\Big).$$
This implies
\begin{eqnarray*}
&&\ch_1(\pi_*E(mH))+\frac{1}{2}\ch_0(\pi_*E(mH))c_1(Y)\\
&=&\pi_*\Big(\ch_2(E(mH))+\frac{1}{2}\ch_1(E(mH))c_1(X)+\frac{r}{12}(c_1^2(X)+c_2(X))\Big).
\end{eqnarray*}
We now compute the terms of the above equation:
\begin{eqnarray*}
\ch_1(E(mH))&=&\ch_1(E)+rmH;\\
\ch_2(E(mH))&=&\ch_2(E)+mH\ch_1(E)+\frac{1}{2}rm^2H^2;\\
\ch_0(\pi_*(E(mH)))&=&\chi(E_y(mH))\\
&=&\deg E_y+rm\deg H_y-r(g-1).
\end{eqnarray*}
It follows that
\begin{eqnarray*}
\ch_1(\pi_*E(mH))&=&\pi_*\left(\frac{r}{12}(c_1^2(X)+c_2(X))+\frac{1}{2}c_1(E)c_1(X)+\ch_2(E)\right)\\
&&+\left(\frac{r}{2}(g-1)-\frac{\deg
E_y}{2}\right)c_1(Y)+\frac{r\pi_*H^2}{2}m^2\\
&&+\left(\pi_*\Big(Hc_1(E)+\frac{r}{2}Hc_1(X)\Big)-\frac{r\deg
H_y}{2}c_1(Y)\right)m.
\end{eqnarray*}
Since
\begin{eqnarray*}
\pi_*(Hc_1(X))-\deg H_yc_1(Y)&=&\pi_*(Hc_1(X))-\pi_*\Big(H\pi^*c_1(Y)\Big)\\
&=&-\pi_*(HK_{X/Y})
\end{eqnarray*}
and
\begin{eqnarray*}
\pi_*(c_1(E)c_1(X))-\deg E_yc_1(Y)&=&\pi_*(c_1(E)c_1(X))-\pi_*\Big(c_1(E)\pi^*c_1(Y)\Big)\\
&=&-\pi_*\Big(c_1(E)K_{X/Y}\Big),
\end{eqnarray*}
one obtains
\begin{eqnarray*}
\ch_1(\pi_*E(mH))&=&\pi_*\Big[\frac{rH^2}{2}m^2+H\big(c_1(E)-\frac{r}{2}K_{X/Y}\big)m+\ch_2(E)\\
&&-\frac{1}{2}c_1(E)K_{X/Y}+\frac{r}{12}(c_1^2(X)+c_2(X))\Big]-\frac{r(g-1)}{2}K_Y.
\end{eqnarray*}
Some other simple computations show that
\begin{eqnarray*}
h^0(E_y(mH))&=&\deg E_y+rm\deg H_y-r(g-1);\\
c_1^2(E(mH))&=&c_1^2(E)+2rmHc_1(E)+r^2m^2H^2;\\
\deg(E_y(mH))&=&\deg E_y+rm\deg H_y.
\end{eqnarray*}
Substituting these equations into the expression for $S_m$, one has
\begin{eqnarray*}
S_m&=&m^2r^3\pi_*\Big((\deg H_y)HK_{X/Y}-(g-1)H^2\Big)\\
&+&mr\deg
H_y\pi_*\left(c_1^2(E)-2r\ch_2(E)+rc_1(E)K_{X/Y}-\frac{r^2}{6}(c_1^2(X)+c_2(X))\right)\\
&+&mr^3(g-1)(\deg H_y)K_Y+mr^2\pi_*\Big((\deg
E_y)HK_{X/Y}-(2g-2)Hc_1(E)\Big)\\
&+&Z,
\end{eqnarray*}
where $Z$ is a $\mathbb{Q}$-divisor on $Y$ which is independent of
$m$. From the positivity of $A_1\cdots A_{n-2}S_m$ for $m\gg0$, one
sees that
$$A_1\cdots A_{n-2}\pi_*\Big((\deg H_y)HK_{X/Y}-(g-1)H^2\Big)\geq0.$$
If the equality holds, we obtain the positivity of the coefficient
of $m$ in $A_1\cdots A_{n-2}S_m$. Hence we are done!
\end{proof}

From Corollary \ref{Slope2}, one can deduce relative Bogomolov's
inequality \cite{Mo1} for trivial fibrations in any characteristic.
We let $C$ be a smooth projective curve of genus $g$ defined over
$k$.
\begin{corollary}\label{relative}
Assume $X=C\times Y$ is a product of $C$ and $Y$ with projections
$\pi:X\rightarrow Y$ and $p:X\rightarrow C$. Let
$A_1,\cdots,A_{n-2}$ be ample divisors on $Y$ and $E$ a torsion free
sheaf on $X$. If $E_y$ is semistable, then
$$\pi^*(A_1\cdots A_{n-2})\Delta(E)\geq0.$$
\end{corollary}
\begin{proof}
In our situation, one sees that $K_X=p^*K_C+\pi^*K_Y$,
$K_{X/Y}=p^*K_C$ and $c_2(X)=\pi^*c_2(Y)+p^*K_C\cdot\pi^*K_Y$. Hence
$$\pi_*\Big(c_1^2(X)+c_2(X)\Big)=3\pi_*\Big(p^*K_C\cdot\pi^*K_Y\Big)=6(g-1)K_Y.$$
Let $c$ be a point of $C$ and let $H=p^*c$. Then one deduces that
$\deg H_y=1$, $H^2=0$ and $HK_{X/Y}=0$. Thus
$$\pi^*(A_1\cdots A_{n-2})\Big((\deg H_y)HK_{X/Y}-(g-1)H^2\Big)=0.$$ By
Corollary \ref{4.3}, one sees that the inequality (\ref{4.3}) holds.
On the other hand, we have
\begin{eqnarray*}
&&A_1\cdots A_{n-2}\pi_*\Big(\big((2g-2)H-K_{X/Y}\big)c_1(E)\Big)\\
&=&A_1\cdots
A_{n-2}\pi_*\Big(\big((2g-2)p^*c-p^*K_C\big)c_1(E)\Big)=0.
\end{eqnarray*}
Therefore, the right hand side of the inequality (\ref{4.3}) is zero
in our case. It follows that $\pi^*(A_1\cdots
A_{n-2})\Delta(E)\geq0$.
\end{proof}

\section{Proof of the main theorems}\label{S5}
The aim of this section is to prove our main theorems from relative
Bogomolov's inequality in Corollary \ref{relative}.

\begin{theorem}\label{Bog}
Let $X=C\times Y$ be the product of a smooth curve $C$ and a $n-1$
dimensional smooth projective variety $Y$ defined over $k$ with
projections $\pi:X\rightarrow Y$ and $p:X\rightarrow C$. Let
$A_1,\cdots, A_{n-2}$ be ample divisors on $Y$, $H$ an ample divisor
on $X$ and $E$ a torsion free sheaf on $X$. If $E$ is
$\mu_{\pi^*A_1,\cdots,\pi^*A_{n-2}, H}$-semistable, then
$$\pi^*(A_1\cdots A_{n-2})\Delta(E)\geq0.$$
\end{theorem}
\begin{proof}
The proof is by induction on the rank of $E$. For rank 1 the
assertion is obvious. Assume that the theorem holds for every
semistable sheaf of rank less than $r$ and $E$ is of rank $r$.

Let $K(Y)$ be the function field of $Y$ and $y$ be a general point
of $Y$. Denote the generic fibre $X\times_Y\Spec(K(Y))$ by
$X_{\eta}$. Let $E_{\eta}$ be the restriction of $E$ to $X_{\eta}$
and $E_y$ the restriction of $E$ to the general fibre $X_y$ of
$\pi$. If $E$ is $\mu_{\pi^*A_1,\cdots,\pi^*A_{n-2},
\pi^*A_1}$-semistable, then $E_{\eta}$ is semistable. By the
openness of semistability (see \cite[Proposition 2.3.1]{HL}), one
sees that $E_y$ is semistable. Hence Corollary \ref{relative}
implies $$\pi^*(A_1\cdots A_{n-2})\Delta(E)\geq0.$$

Now we assume that $E$ is not $\mu_{\pi^*A_1,\cdots,\pi^*A_{n-2},
\pi^*A_1}$-semistable. By \cite[Lemma 4.C.5]{HL} $\footnote{This
lemma is stated for the surface case, but its proof still works for
our situation.}$, there is a non-negative rational number $t$ and a
saturated subsheaf $E_0\subset E$ with $\rank E_0=r_0$ such that
$$\mu_{\pi^*A_1,\cdots,\pi^*A_{n-2},
\pi^*A_1}(E_0)>\mu_{\pi^*A_1,\cdots,\pi^*A_{n-2}, \pi^*A_1}(E),$$
and $E$ and $E_0$ are
$\mu_{\pi^*A_1,\cdots,\pi^*A_{n-2},H_t}$-semistable of the same
slope, where
$$H_t=H+t\pi^*A_1.$$
Since $E_0$ is saturated, one sees that $E_1=E/E_0$ is torsion free
and $\mu_{\pi^*A_1,\cdots,\pi^*A_{n-2},H_t}$-semistable of rank
$r_1=r-r_0$. Set $\xi=rc_1(E_0)-r_0c_1(E)$. Then
$$\pi^*(A_1\cdots A_{n-2})H_t\xi=0.$$ It follows from the Hodge index theorem
that $$\pi^*(A_1\cdots A_{n-2})\xi^2\leq0.$$ Moreover, the following
identity holds:
$$\pi^*(A_1\cdots A_{n-2})\Big(\Delta(E)-\frac{r}{r_0}\Delta(E_0)-\frac{r}{r_1}\Delta(E_1)\Big)=-\frac{\pi^*(A_1\cdots A_{n-2})\xi^2}{r_0r_1}\geq0.$$
By our induction assumption, one has $\pi^*(A_1\cdots
A_{n-2})\Delta(E_i)\geq0$ for $i=0, 1$ and therefore
$$\pi^*(A_1\cdots A_{n-2})\Delta(E)\geq0.$$
\end{proof}

By this, we immediately deduce Theorem \ref{Surface}.
\begin{corollary}[=Theorem \ref{Surface}]
Let $H$ be a numerically nontrivial nef divisor on a smooth
projective surface $S$ which is birational to a product type
surface. Then for any $\mu_H$-semistable sheaf $E$ on $S$, we have
$\Delta(E)\geq0$.
\end{corollary}
\begin{proof}
By \cite[Proposition 6.2 and 6.5]{Langer3}, one can assume that $H$
is ample and $S$ is a product type surface. Let $f:C_1\times
C_2\rightarrow S$ be the finite separable morphism associated with
$S$, where $C_1$ and $C_2$ are curves. It turns out that $f^*E$ is
$\mu_{f^*H}$-semistable. It follows from Theorem \ref{Bog} that
$\Delta(f^*E)\geq0$. Thus $\Delta(E)\geq0$.
\end{proof}

Theorem \ref{Surface} can be generalized to high dimensional product
type varieties. Let $X_n=C_1\times\cdots\times C_n$ be a product of
smooth projective curves defined over $k$ with projections
$p_i:X_n\rightarrow C_{i}$ for $i=1,\cdots,n$. We now prove Theorem
\ref{main}.

\begin{theorem}[=Theorem
\ref{main}] Let $X$ be a product type variety with respect to a
finite separable morphism $f:X_n\rightarrow X$. Let $H$ be a product
type ample divisor on $X$. Then for any $\mu_{H}$-semistable sheaf
$E$ on $X$, we have
$$H^{n-2}\Delta(E)\geq0.$$
\end{theorem}
\begin{proof}
Since the semistability is invariant under the pull back of $f$, we
can assume that $X=X_n$ and $f$ is the identity.

For $n\geq2$, let $P(r,n)$ denote the statement: for any product
type ample divisor $H$ on $X_n$, one has
$$H^{n-2}\Delta(E)\geq0$$ for any
$\mu_{H}$-semistable sheaf $E$ of rank $\leq r$ on $X_n$.

Obviously, by Theorem \ref{Surface}, $P(r,2)$ and $P(1,n)$ hold for
any positive integers $n\geq2$ and $r$. By induction one sees that
if $P(r-1,n)$ and $P(r,n-1)$ imply $P(r, n)$, then $P(r, n)$ holds
for any positive integers $n\geq2$ and $r$.

Now we assume $P(r-1,n)$ and $P(r,n-1)$ hold. We let $H$ be a
product type ample divisor on $X_n$ and $F_j$ the general fiber of
$p_j:X_n\rightarrow C_j$. Let $E$ be a $\mu_{H}$-semistable sheaf of
rank $r$ on $X_n$. If $E$ is not $\mu_{F_j,H,\cdots, H}$-semistable
for some $1\leq j\leq n$, then similar to the proof of Theorem
\ref{Bog}, there is a non-negative rational number $t$ and a
saturated subsheaf $E_0\subset E$ with $\rank E_0=r_0$ such that
$$\mu_{F_j,H,\cdots, H}(E_0)>\mu_{F_j,H,\cdots,
H}(E),$$ and $E$ and $E_0$ are $\mu_{H_t,H,\cdots, H}$-semistable of
the same slope, where $$H_t=H+tF_j.$$ Since $E_0$ is saturated, one
sees that $E_1=E/E_0$ is torsion free and
$\mu_{H_t,H\cdots,H}$-semistable of rank $r_1=r-r_0$. Set
$\xi=rc_1(E_0)-r_0c_1(E)$. Then
$$H^{n-2}H_t\xi=0.$$ It follows from the Hodge index theorem
that $$H^{n-2}\xi^2\leq0.$$ Moreover, the following identity holds:
$$H^{n-2}\Big(\Delta(E)-\frac{r}{r_0}\Delta(E_0)-\frac{r}{r_1}\Delta(E_1)\Big)=-\frac{H^{n-2}\xi^2}{r_0r_1}\geq0.$$
From our induction assumptions, it follows that
$H^{n-2}\Delta(E_i)\geq0$ for $i=0, 1$ and therefore
$$H^{n-2}\Delta(E)\geq0.$$

Now we assume that $E$ is $\mu_{F_j,H,\cdots, H}$-semistable for any
$1\leq j\leq n$, then by the openness of semistability, one sees
that $E|_{F_j}$ is $\mu_{H'}$-semistable, here $H'=H|_{F_j}$. Since
$$F_j\cong C_1\times\cdots \times C_{j-1}\times
C_{j+1}\times\cdots\times C_n,$$ by the induction assumptions, we
have $(H')^{n-3}\Delta(E|_{F_j})\geq0$, i.e.,
$$F_jH^{n-3}\Delta(E)\geq0,$$ for $1\leq j\leq n$. Since
$H$ is of product type, we can write
$$H=p_1^*B_1+\cdots+p_n^*B_n,$$ where $B_i$ is a divisor on $C_i$
with $\deg B_i=b_i>0$. Then one concludes
$$H^{n-2}\Delta(E)=\sum_{j=1}^n(b_jF_j)H^{n-3}\Delta(E)\geq0.$$ Thus we are done!
\end{proof}

\section{Applications of Bogomolov's inequality}\label{S6}
In this section we exhibit the applications of Theorem \ref{Surface}
to the positivity of linear systems and torsion free sheaves and the
construction of Bridgeland stability conditions. We always let $S$
be a smooth projective surface defined over $k$ which is birational
to a product type surface in this section.

\subsection{Adjoint linear systems}
\begin{theorem}\label{Reider}
Let $d\geq1$ be an integer and $L$ be a nef divisor on $S$ with
$L^2>4d$. If $|K_S+L|$ is not $(d-1)$-very ample, then there exists
a curve $D$ on $S$ such that $$LD-d\leq D^2<\frac{1}{2}LD<d.$$
\end{theorem}
\begin{proof}
By Theorem \ref{Surface}, the proof is the same as that of
\cite{Reider} and \cite{BS}.
\end{proof}

\subsection{Positivity of semistable sheaves}
Let $H$ be an ample divisor and $E$ a $\mu_H$-semistable torsion
free sheaf on $S$ with $\rank E\geq2$. We define the generalized
discriminant of $E$ to be
$$\overline{\Delta}_H(E):=(H\ch_1(E))^2-2H^2\ch_0(E)\ch_2(E).$$
By Theorem \ref{Surface} and Hodge index theorem, one sees that
$\overline{\Delta}_H(E)\geq0$.

\begin{theorem}\label{Sun}
If $l>(\overline{\Delta}_H(E)-\mu_H(E))/H^2$, then we have
$H^1\big(E(K_S+lH)\big)=0$, and $E(K_S+lH)$ is generated by global
sections if $l>2\rank E+(\overline{\Delta}_H(E)-\mu_H(E))/H^2$.
\end{theorem}
\begin{proof}
The first assertion is just \cite[Corollary 1.8]{Sun}. For the
second assertion, we consider the short exact sequence
$$0\rightarrow K\rightarrow E\xrightarrow{f}
\mathcal{O}_x\rightarrow0,$$ where $x$ is a point in $S$, $f$ is any
surjection and $K=\ker f$. One sees that $K$ is also
$\mu_H$-semistable and
$$\overline{\Delta}_H(E)=\overline{\Delta}_H(K)-2H^2\rank E.$$
Hence by the first assertion, we conclude that $H^1(K(K_S+lH))=0$ if
$$l>2\rank E+(\overline{\Delta}_H(E)-\mu_H(E))/H^2.$$  This implies
the induced map $H^0(E(K_S+lH))\rightarrow H^0(\mathcal{O}_x)$ is
surjective for any $x\in S$ and any surjection $f: E\rightarrow
\mathcal{O}_x$. Therefore $E(K_S+lH)$ is generated by global
sections.
\end{proof}

\subsection{Bridgeland stability conditions on surfaces}
The notion of Bridgeland stability condition was introduced in
\cite{Bri1}. In recent years, this stability condition has drawn a
lot of attentions, and has been investigated intensively. Let $X$ be
a smooth projective variety defined over $k$ with $\dim X=n$, and
let $H$ be an ample divisor on $X$.

\begin{definition}
A Bridgeland stability condition on $X$ is a pair $\sigma=(Z,
\mathcal{A})$, where where $\mathcal{A}$ is the heart of a bounded
$t$-structure on $\D^b(X)$, and $Z:\K(\D^b(X))\rightarrow
\mathbb{C}$ is a group homomorphism (called central charge) such
that
\begin{itemize}
\item $Z$ satisfies the following positivity property for any non-zero $E\in \mathcal{A}$:
$$Z(E)\in\{re^{i\pi\phi}: r>0, 0<\phi\leq1\}.$$
\item Every
non-zero object in $\mathcal{A}$ has a Harder-Narasimhan filtration
in $\mathcal{A}$ with respect to $\nu_Z$-stability, here the slope
$\nu_Z$ of an object $E\in \mathcal{A}$ is defined by
\begin{eqnarray*}
\nu_{Z}(E)= \left\{
\begin{array}{lcl}
+\infty,  & &\mbox{if}~\Im Z(E)=0,\\
&&\\
-\frac{\Re Z(E)}{\Im Z(E)}, & &\mbox{otherwise}.
\end{array}\right.
\end{eqnarray*}
\end{itemize}
\end{definition}

We now review the construction of Bridgeland stability condition in
\cite{Bri2, AB}. For a fixed $\mathbb{Q}$-divisor $D$ on $X$, we
define the twisted Chern character $\ch^{D}=e^{-D}\ch$. More
explicitly, we have
\begin{eqnarray*}
\begin{array}{lcl}
\ch^{D}_0=\ch_0=\rank  && \ch^{D}_2=\ch_2-D\ch_1+\frac{D^2}{2}\ch_0\\
&&\\
\ch^{D}_1=\ch_1-D\ch_0 &&
\ch^{D}_3=\ch_3-D\ch_2+\frac{D}{2}\ch_1-\frac{D^3}{6}\ch_0.
\end{array}
\end{eqnarray*}

We define the twisted slope $\mu_{H, D}$ of a coherent sheaf $E\in
\Coh(X)$ by
\begin{eqnarray*}
\mu_{H, D}(E)= \left\{
\begin{array}{lcl}
+\infty,  & &\mbox{if}~\ch^D_0(E)=0,\\
&&\\
\frac{H^{n-1}\ch_1^{D}(E)}{H^n\ch_0^{D}(E)}, & &\mbox{otherwise}.
\end{array}\right.
\end{eqnarray*}
Similarly as Definition \ref{def2.1}, one can define the stability
for sheaves with respect to $\mu_{H, D}$. Let $\alpha>0$ and $\beta$
be two real numbers. There exists a \emph{torsion pair}
$(\mathcal{T}_{\beta H+D},\mathcal{F}_{\beta H+D})$ in $\Coh(X)$
defined as follows:
\begin{eqnarray*}
\mathcal{T}_{\beta H+D}=\{E\in\Coh(X):\mu^-_{H, \beta H+D}(E)>0 \}\\
\mathcal{F}_{\beta H+D}=\{E\in\Coh(X):\mu^+_{H, \beta H+D}(E)\leq0
\}
\end{eqnarray*}
Equivalently, $\mathcal{T}_{\beta H+D}$ and $\mathcal{F}_{\beta
H+D}$ are the extension-closed subcategories of $\Coh(X)$ generated
by $\mu_{H, \beta H+D}$-semistable sheaves of positive and
non-positive slope, respectively.

\begin{definition}
We let $\Coh^{\beta H+D}(X)\subset \D^b(X)$ be the extension-closure
$$\Coh^{\beta H+D}(X)=\langle\mathcal{T}_{\beta H+D}, \mathcal{F}_{
\beta H+D}[1]\rangle.$$
\end{definition}

By the general theory of torsion pairs and tilting \cite{HRS},
$\Coh^{\beta H+D}(X)$ is the heart of a bounded t-structure on
$\D^b(X)$; in particular, it is an abelian category. Consider the
following central charge
$$Z_{\alpha, \beta}(E)=H^{n-2}\Big(\frac{\alpha^2 H^2}{2}\ch_0^{\beta H+D}(E)-\ch_2^{\beta H+D}(E)+i H\ch_1^{\beta H+D}(E)
\Big).$$

\begin{theorem}\label{Bri}
If $X$ is birational to a surface of product type, then for any
$(\alpha, \beta)\in \mathbb{R}_{>0}\times\mathbb{R}$,
$\sigma_{\alpha, \beta}=(Z_{\alpha, \beta}, \Coh^{\beta H+D}(X))$ is
a Bridgeland stability condition.
\end{theorem}
\begin{proof}
The assumption on $X$ guarantees that Bogomolov's inequality holds
on it. Hence the required assertion is proved in \cite{Bri2, AB}.
\end{proof}

\section{Counterexamples to Bogomolov's inequality}\label{S7}
In this section we exhibit some new counterexamples to Bogomolov's
inequality in positive characteristic. In \cite{Ray, Muk}, the
authors constructed some surfaces in positive characteristic with
$c_1^2>0$ and $c_2<0$ on which Kodaira's vanishing fails. These
surfaces give rise to rank two semistable sheaves violating
Bogomolov's inequality.

We now construct some high rank counterexamples to Bogomolov's
inequality. Let $X$ be a smooth projective surface defined over $k$
and $H$ an ample divisor on $X$. Assume that $\chr(k)=p>0$. Denote
by $F$ the absolute Frobenius morphism of $X$.

\begin{theorem}\label{Count}
Assume that $K_XH>0$, $K_X^2>2c_2(X)$ and $\Omega_X^1$ is
$\mu_H$-semistable. Let $\mathcal{L}$ be a line bundle on $X$. Then
$F_*\mathcal{L}$ is $\mu_H$-semistable but
$\Delta(F_*\mathcal{L})<0$.
\end{theorem}
\begin{proof}
By \cite[Theorem 5.1]{KS}, one sees that $F_*\mathcal{L}$ is
$\mu_H$-semistable (see also \cite{SunX}). It remains to compute
$\Delta(F_*\mathcal{L})$ explicitly.

From the Grothendieck-Riemann-Roch theorem, it follows that
$$\ch(F_*\mathcal{L})\td(X)=F_*\big(\ch(\mathcal{L})\td(X)\big).$$
Since $\td(X)=1+\frac{1}{2}c_1+\frac{1}{12}(c_1^2+c_2)$, the above
equation implies
$$
\frac{1}{2}\ch_0(F_*\mathcal{L})c_1+\ch_1(F_*\mathcal{L})=F_*\left(\frac{c_1}{2}+c_1(\mathcal{L})
\right)=p\left(\frac{c_1}{2}+c_1(\mathcal{L}) \right)
$$
and
$$
\frac{c_1^2+c_2}{12}\ch_0(F_*\mathcal{L})+\frac{c_1}{2}\ch_1(F_*\mathcal{L})+\ch_2(F_*\mathcal{L})
=\frac{c_1^2+c_2}{12}+\frac{c_1}{2}c_1(\mathcal{L})+\ch_2(\mathcal{L}).
$$
A simple computation shows $\ch_0(F_*\mathcal{L})=p^2$,
$$\ch_1(F_*\mathcal{L})=\frac{p^2-p}{2}K_X+pc_1(\mathcal{L})$$ and
$$\ch_2(F_*\mathcal{L})=\frac{1-p^2}{12}(K_X^2+c_2)+\frac{p^2-p}{4}K_X^2+\frac{p-1}{2}K_Xc_1(\mathcal{L})+\frac{c_1^2(\mathcal{L})}{2}.$$
Therefore one concludes that
$$\Delta(F_*\mathcal{L})=\frac{p^4-p^2}{12}(2c_2-K_X^2)<0.$$
\end{proof}
To get a surface satisfying the hypotheses of Theorem \ref{Count},
let $S$ be a smooth complex projective surface with ample $K_S$ and
$K_S^2>2c_2(S)$. It is well known that $\Omega^1_S$ is
$\mu_{K_S}$-stable (see \cite{Tsu} for example). By the standard
spreading out technique, we have a subring $R\subset \mathbb{C}$,
finitely generated over $\mathbb{Z}$, and a scheme
$\pi:S_R\rightarrow \spec R$ so that $\pi$ is smooth and projective
and $S=S_R\times_{R}\mathbb{C}$. By the openness of ampleness and
stability, one sees that $K_{S_m}$ is ample and $\Omega^1_{S_m}$ is
$\mu_{K_{S_m}}$-semistable for a general maximal ideal $m\in\spec
R$, where $S_m=S_R\times_{R}\overline{(R/m)}$ is the geometric fibre
of $\pi$ over $m$. One notes that $\chr (R/m)>0$. Hence $S_m$
satisfies the hypotheses of Theorem \ref{Count}.

\bibliographystyle{amsplain}

\end{document}